\documentclass[11pt, reqno]{amsart}
\usepackage{amsmath, amsthm, amscd, amsfonts, amssymb, graphicx, color}
\usepackage[bookmarksnumbered, colorlinks, plainpages]{hyperref}
\input{mathrsfs.sty}

\newtheorem{theorem}{Theorem}[section]
\newtheorem{lemma}[theorem]{Lemma}
\newtheorem{proposition}[theorem]{Proposition}
\newtheorem{corollary}[theorem]{Corollary}
\theoremstyle{definition}

\newtheorem{example}[theorem]{Example}

\theoremstyle{remark}
\newtheorem{remark}[theorem]{Remark}

\begin{document}
\title[operator log-convex functions ]{  operator log-convex functions and $f$-divergence functional }

\author[ M. Kian]{ Mohsen Kian}

\address{Mohsen Kian:\ \
School of Mathematics, Institute for Research in Fundamental Sciences (IPM), P.O. Box: 19395-5746, Tehran, Iran.}

 \address{Department of Mathematics, Faculty of Basic Sciences, University of Bojnord, P. O. Box
1339, Bojnord 94531, Iran.}
\email{ kian@member.ams.org and kian@ub.ac.ir.}



\begin{abstract}
 We present a characterization of operator log-convex functions by using positive linear mappings. More precisely, we show that the continuous function $f:(0,\infty)\to(0,\infty)$ is operator log-convex if and only if $f(\Phi(A))\leq(\Phi(f(A)^{-1}))^{-1}$ for every  strictly positive operator $A$ and every unital positive linear map $\Phi$. Moreover, we study  the  non-commutative $f$-divergence functional of operator log-convex functions. In particular, we prove  that $f$ is operator log-convex if and only if  the non-commutative $f$-divergence functional  is operator log-convex in its first variable and operator convex in its second variable.

\end{abstract} \maketitle

\section{Introduction and Preliminaries }
Throughout this paper, assume that $\mathbb{B}(\mathscr{H})$ is the  $C^*$-algebra of all bounded linear operators on a Hilbert space $\mathscr{H}$ and $I$ denote the identity operator. An operator $A\in\mathbb{B}(\mathscr{H})$ is called positive (denoted by $A\geq 0$) if $\langle Ax,x\rangle\geq 0$ for every $x\in\mathscr{H}$. If in addition $A$ is invertible, then it is called strictly positive  (denoted by $A>0$). A linear map $\Phi$ on $\mathbb{B}(\mathscr{H})$ is said  to be positive if $\Phi(A)\geq  0$ whenever $A\geq0$ and is called unital if $\Phi(I)=I$.

A continuous real  function $f$ defined on an interval $J$ is
said to be operator convex if
\begin{eqnarray*}
 f(\lambda A+(1-\lambda)B)\leq \lambda f(A)+(1-\lambda)f(B)
\end{eqnarray*}
for all self-adjoint operators $A,B$ with spectra contained in $J$ and every $\lambda\in[0,1]$. If $-f$ is operator convex, then $f$ is said to be operator concave. If $f:J\to \mathbb{R}$ is operator convex, then the celebrated  Hansen--Pedersen--Jensen operator inequality \cite{HP}  $f(C^*AC)\leq C^*f(A)C$ holds true for every self-adjoint operator $A$ with spectrum  contained in $J$ and every isometry $C$. Another variant of  this inequality,  the
Choi--Davis--Jensen inequality asserts that if $f$ is operator convex, then
\begin{align}\label{cdj}
  f(\Phi(A))\leq \Phi(f(A))
\end{align}
for every unital positive linear map $\Phi$ on $\mathbb{B}(\mathscr{H})$ (see e.g. \cite{Fu}).  The reader is referred to \cite{Fu,HPP,IMP, KM} and references therein for more information about operator convex functions and the  Jensen operator inequality.

 Let $\mathfrak{A}$ and $\mathfrak{B}$ be $C^*$-algebras of Hilbert space operators and   $T$ be a locally compact Hausdorff space with a bounded Radon measure $\mu$.  A field $(A_t)_{t\in T}$
of operators in $\mathfrak{A}$ is said to be continuous if the function $t\mapsto A_t$ is norm continuous on $T$. Moreover, if  the function $t\mapsto A_t$ is integrable on $T$, then the Bochner integral $\int_TA_td\mu(t)$ is
defined to be the unique element of $\mathfrak{A}$ for which
\begin{eqnarray*}
\rho\left(\int_TA_td\mu(t)\right)=\int_T\rho(A_t)d\mu(t),
\end{eqnarray*}
for any linear functional $\rho$ in the norm dual $\mathfrak{A}^*$ of $\mathfrak{A}$.

  A field $(\Phi_t)_{t\in T}:\mathfrak{A}\to\mathfrak{B}$ of positive linear mappings is said to be continuous if the function $t\mapsto\Phi_t(A)$ is continuous on $T$ for every $A\in\mathfrak{A}$. If the $C^*$-algebras $\mathfrak{A}$ and
$\mathfrak{B}$ are unital and the function $t\mapsto \Phi_t(I)$ is integrable on $T$ with integral $I$, then we say that the field $(\Phi_t)_{t\in T}$ is unital.

By the well-known Kubo--Ando theory \cite{KA}, an operator mean $\sigma$ is a binary operation on the set of positive operators which satisfies the following conditions:
\begin{enumerate}\label{prt}
  \item monotonicity:  \ \ if $A\leq C$ and $B\leq D$, then $A\sigma B\leq C\sigma D$;
  \item Transformer inequality: \ \ $C(A\sigma B)C\leq (CAC)\sigma(CBC)$. If $C$ is invertible, then equality holds.
      \item Continuity:  \ \ if $A_n$ and $B_n$ are two decreasing sequences of positive operators which are converging respectively  to $A$ and $B$ in the strong operator topology, then $A_n\sigma B_n$ converges  to $A\sigma B$.
\end{enumerate}
Kubo and Ando \cite{KA} showed that for every operator mean $\sigma$ there exists an operator monotone function $f:(0,\infty)\to (0,\infty)$ such that
\begin{align}\label{mean}
A\sigma B=B^{\frac{1}{2}}f\left(B^{\frac{-1}{2}}AB^{\frac{-1}{2}}\right)B^{\frac{1}{2}}
\end{align}
for all positive operators $A,B$. Conversely, they proved that  if $f:(0,\infty)\to (0,\infty)$ is operator monotone, the binary operation  defined by \eqref{mean} is an operator mean.  Some of the  most familiar  operator means are
\begin{align*}
  A\nabla B&=\frac{A+B}{2}\qquad\qquad\qquad\qquad \mbox{ Arithmetic mean}\\
   A\sharp B&=B^{\frac{1}{2}}\left(B^{\frac{-1}{2}}AB^{\frac{-1}{2}}\right)^{\frac{1}{2}}B^{\frac{1}{2}}\quad\quad \mbox{Geometric mean}\\
   A!B&=\left(\frac{A^{-1}+B^{-1}}{2}\right)^{-1}\qquad\qquad \mbox{Harmonic mean}.
   \end{align*}

A continuous real function $f:(0,\infty)\to(0,\infty)$ is called operator log-convex function if
$$f(A\nabla B)\leq f(A)\sharp f(B)$$
for all positive operators $A$ and $B$. This notion was considered   by Ando and Hiai \cite{AH}.
They presented the following result.\\
\textbf{Theorem A.}\cite[Theorem 2.1]{AH} Let $f:(0,\infty)\to(0,\infty)$ be continuous. The following conditions are equivalent:
\begin{enumerate}
  \item $f$ is operator decreasing;
  \item $f$ is operator log-convex;
  \item $f(A\nabla B)\leq f(A)\sigma f(B)$ for all positive operators $A$ and $B$ and every operator mean $\sigma$;
  \item $f(A\nabla B)\leq f(A)\sigma f(B)$ for all positive operators $A$ and $B$ and some operator mean $\sigma\neq\nabla$.
\end{enumerate}

If $f$ is operator convex, then \eqref{mean}  defines  \cite{G,E}   the perspective of $f$  denoted by $g$, i.e.,
$$g(A,B)=B^{\frac{1}{2}}f\left(B^{\frac{-1}{2}}AB^{\frac{-1}{2}}\right)B^{\frac{1}{2}}.$$
It is known  that $f$ is operator convex if and only if $g$ is jointly  operator convex \cite{G,E}. A more general version of  $g$, the non-commutative $f$-divergence functional $\Theta$ was defined in \cite{MK} to be
\begin{eqnarray*}
 \Theta(\widetilde{A},\widetilde{B})=\int_TB_t^{\frac{1}{2}}f\left(B_t^{-\frac{1}{2}}
 A_tB_t^{-\frac{1}{2}}\right)B_t^{\frac{1}{2}}d\mu(t),
\end{eqnarray*}
where $\widetilde{A}=(A_t)_{t\in T}$ and $\widetilde{B}=(B_t)_{t\in T}$ are continuous fields of strictly  positive operators in $\mathfrak{A}$.

In section 2, we give some properties of operator log-convex functions. We present a characterization of operator log-convex functions by using positive linear mappings. More precisely, we show that the continuous function $f:(0,\infty)\to(0,\infty)$ is operator log-convex if and only if $f(\Phi(A))\leq(\Phi(f(A)^{-1}))^{-1}$ for every strictly positive operator $A$ and every unital positive linear map $\Phi$.

In section 3, we study    the non-commutative $f$-divergence functional of operator log-convex functions. In particular, we prove that $f$ is operator log-convex if and only if $\Theta$ is operator log-convex in its first variable and operator convex in its second variable.
\section{Main Result}

  If $f$ is operator log-convex, then a sharper variant of the Jensen operator inequality holds true. The proof of the next lemma is based on that of \cite[Theorem 1.9]{Fu}.
  \begin{lemma}\label{jl}
    If $f:(0,\infty)\to(0,\infty)$ is an operator log-convex function,  then
    $$f\left(C^*AC\right)\leq \left(C^*f(A)^{-1}C\right)^{-1}$$
   for every  strictly positive operator $A$ and  every isometry $C$ provided that $C^*f(A)^{-1}C$ is invertible.
  \end{lemma}
  \begin{proof}
  If $A$ and $B$ are two strictly positive operators in $\mathbb{B}(\mathscr{H})$,  then  $X=
\left(
\begin{array}{cc}
A&0 \\
0&B
\end{array}\right)$ is regarded as a strictly positive operator in $\mathbb{B}(\mathscr{H}\oplus\mathscr{H})$.  Set  $D=\sqrt{I-CC^*}$ so that  operators $U$ and $V$  defined by
 \begin{eqnarray*}
U=\left(
\begin{array}{cc}
C&D \\
0&-C^*
\end{array}\right),\qquad V=\left(
\begin{array}{cc}
C&-D \\
0&C^*
\end{array}\right),
\end{eqnarray*}
are unitary operators in $\mathbb{B}(\mathscr{H}\oplus\mathscr{H})$
and
\begin{align*}
U^*XU=\left(
\begin{array}{cc}
C^*AC&C^*AD \\
DAC &DAD+CBC^*
\end{array}\right),\qquad V^*XV=\left(
\begin{array}{cc}
C^*AC&-C^*AD  \\
-DAC&DAD+CBC^*
\end{array}\right).
\end{align*}
Therefore
\begin{align*}
\left(
\begin{array}{cc}
f(C^*AC)&0 \\
0 &f(DAD+CBC^*)
\end{array}\right)&=f\left(
\begin{array}{cc}
C^*AC&0 \\
0 &DAD+CBC^*
\end{array}\right)\\
&= f\left(\frac{U^*XU+V^*XV}{2}\right)\\
&\leq  f(U^*XU)!f(V^*XV) \quad (\mbox{since $f$ is operator log-convex})\\
& = (U^*f(X)U)!(V^*f(X)V)\quad (\mbox{since $U,V$ are unitary})\\
 & = \left(\frac{(U^*f(X)U)^{-1}+(V^*f(X)V)^{-1}}{2}\right)^{-1}\\
&=\left(\frac{U^*f(X)^{-1}U+V^*f(X)^{-1}V}{2}\right)^{-1}\\
&= \left(
\begin{array}{cc}
C^*f(A)^{-1}C&0 \\
0 &Df(A)^{-1}D+Cf(B)^{-1}C^*
\end{array}\right)^{-1}.
\end{align*}
Hence $f(C^*AC)\leq\ (C^*f(A)^{-1}C)^{-1}$.
  \end{proof}

  Note that the operator convexity of $f(x)=x^{-1}$ implies that
  $$f(C^*AC)\leq\ (C^*f(A)^{-1}C)^{-1}\leq C^*f(A)C.$$
\begin{corollary}\label{co11}
  If $f:(0,\infty)\to(0,\infty)$ is an operator log-convex function  and $A_1,\cdots,A_n$ are strictly positive operators,  then $$f\left(\sum_{i=1}^{n}C_i^*A_iC_i\right)\leq \left(\sum_{i=1}^{n}C_i^*f(A_i)^{-1}C_i\right)^{-1}$$
  for all  operators $C_i$\ $(i=1,\cdots,n)$ with  $\sum_{i=1}^{n}C_i^*C_i=I$.
\end{corollary}
\begin{proof}
  Apply  Lemma \ref{jl} to the strictly positive operator $A=A_1\oplus\cdots\oplus A_n$ and the isometry $C=\left(
\begin{array}{c}
C_1\\
\vdots\\
C_n
\end{array}\right)$.
\end{proof}

We can present the following  characterization of operator log-convex functions using positive linear mappings.
  \begin{theorem}\label{tcdjl}
 A continuous function $f:(0,\infty)\to(0,\infty)$ is operator log-convex if and only if
 \begin{align}\label{cdjl}
f(\Phi(A))\leq \Phi\left(f\left(A\right)^{-1}\right)^{-1}
 \end{align}
 for every unital positive linear map $\Phi$ and every strictly positive operator $A$.
  \end{theorem}
 \begin{proof}
   Suppose that $A$ is a strictly positive operator an a finite dimensional Hilbert space $\mathscr{H}$ with the spectral decomposition $A=\sum_{i=1}^{n}\lambda_iP_i$. If $\Phi$ is a unital positive linear map on $\mathbb{B}(\mathscr{H})$, then $\Phi(A)=\sum_{i=1}^{n}\lambda_i\Phi(P_i)$ and $\sum_{i=1}^{n}\Phi(P_i)=I$. Therefore
   \begin{align*}
     f(\Phi(A))=f\left(\sum_{i=1}^{n}\lambda_i\Phi(P_i)\right)  &=f\left(\sum_{i=1}^{n}\Phi(P_i)^{\frac{1}{2}}\lambda_i\Phi(P_i)^{\frac{1}{2}}\right)\\
     &\leq \left(\sum_{i=1}^{n}\Phi(P_i)^{\frac{1}{2}}f(\lambda_i)^{-1}\Phi(P_i)^{\frac{1}{2}}\right)^{-1}
     \quad(\mbox{by Corollary \ref{co11}})\\
     &=\left(\sum_{i=1}^{n}f(\lambda_i)^{-1}\Phi(P_i)\right)^{-1}\\
     &=\Phi\left(f(A)^{-1}\right)^{-1}.
   \end{align*}
   If $A$ is a strictly positive operator on an infinite dimensional Hilbert space, then \eqref{cdjl}  follows by using a continuity argument.

 For the converse assume that \eqref{cdjl} holds true.
  put
  $$\mathfrak{D}(\mathscr{H}\oplus\mathscr{H})=\left\{\left(
\begin{array}{cc}
A&0 \\
0&B
\end{array}\right);\quad A,B\in \mathbb{B}(\mathscr{H})\right\}.$$
  Then $\mathfrak{D}(\mathscr{H}\oplus\mathscr{H})$ is a unital closed  $*$-subalgebra of $\mathbb{B}(\mathscr{H}\oplus\mathscr{H})$.
    Let the unital positive linear map $\Psi$ be defined on $\mathfrak{D}(\mathscr{H}\oplus\mathscr{H})$ by
 $$\Psi\left(\left(
\begin{array}{cc}
A&0 \\
0&B
\end{array}\right)\right)=\frac{A+B}{2}.$$
Now if $A$ and $B$ are two strictly positive operators on $\mathscr{H}$, then  it follows from \eqref{cdjl} that
\begin{align*}
  f(A\nabla B)&=f\left(\Psi\left(\left(
\begin{array}{cc}
A&0 \\
0&B
\end{array}\right)\right)\right)\leq \Psi\left( f\left(
\begin{array}{cc}
A&0 \\
0&B
\end{array}\right)^{-1}\right)^{-1}\\
&=\Psi\left( \left(
\begin{array}{cc}
f(A)^{-1}&0 \\
0&f(B)^{-1}
\end{array}\right)\right)^{-1}
=\left(\frac{f(A)^{-1}+f(B)^{-1}}{2}\right)^{-1}\\
&=f(A)!f(B)\leq f(A)\sharp f(B),
\end{align*}
Which implies that $f$ is operator log-convex.
 \end{proof}
Note that it follows from the operator convexity of $x\to x^{-1}$ that
$$f(\Phi(A))\leq \Phi\left(f\left(A\right)^{-1}\right)^{-1}\leq \Phi(f(A)).$$
Let us give an example to show that in the case of operator log-convex functions, inequality \eqref{cdjl} is really sharper than the Choi--Davis--Jensen inequality.
\begin{example}
The function $f(x)=x^{\frac{-1}{2}}$ is operator log-convex on $(0,\infty)$. Assume that the unital  positive linear map $\Phi:\mathcal{M}_3(\mathbb{C})\to\mathcal{M}_2(\mathbb{C})$ is defined by
$$\Phi((a_{ij}))=(a_{ij})_{2\leq i,j\leq3}.$$
If $A\in\mathcal{M}_3(\mathbb{C})$ is the positive matrix
$$A=\left(
\begin{array}{ccc}
2&0&1 \\
0&1 & 1\\
1&1&3
\end{array}\right),$$
then by a simple calculation we have
\begin{align*}
f(\Phi(A))&=\left(
\begin{array}{cc}
1.1945&-0.2706 \\
-0.2706 & 0.6533
\end{array}\right),\qquad \Phi\left(f\left(A\right)^{-1}\right)^{-1}=\left(
\begin{array}{cc}
1.2192   &-0.2933\\
   -0.2933&    0.6760
\end{array}\right)\\
\Phi(f(A))&=\left(
\begin{array}{cc}
    1.2420 &  -0.3261\\
   -0.3261 &   0.7234
\end{array}\right)
\end{align*}
and so
$$f(\Phi(A))\lvertneqq \Phi\left(f\left(A\right)^{-1}\right)^{-1}\lvertneqq \Phi(f(A)).$$
\end{example}
\begin{corollary}
  If $\Phi$ is a unital positive linear map and $A$ is a strictly positive operator, then
 \begin{enumerate}
   \item $\Phi(A)^{-\alpha}\leq\Phi\left(A^{\alpha}\right)^{-1}\leq\Phi\left(A^{-\alpha}\right)$ \quad for all $0\leq \alpha\leq 1$.
   \item $\Phi\left(A^{\alpha}\right)^{\frac{1}{\alpha}}\leq\Phi\left(A^{-1}\right)^{-1}\leq \Phi\left(A\right)$ \quad for all $\alpha\leq -1$.
 \end{enumerate}
\end{corollary}
\begin{proof}
$(1)$: follows from the operator log-convexity of $t^{-\alpha}$.\\
$(2)$: the function $t\to t^{\frac{1}{\alpha}}$ is   operator log-convex. So it follows from Theorem \ref{tcdjl} that $\Phi(A)^{\frac{1}{\alpha}}\leq \left(\Phi\left(A^{\frac{-1}{\alpha}}\right)\right)^{-1}$. Replacing $A$ by $A^{\alpha}$ we get desired inequality.
\end{proof}

\begin{corollary}
 Let $\Phi_1,\cdots,\Phi_n$ be positive linear mappings on $\mathbb{B}(\mathscr{H})$ such that $\sum_{i=1}^{n}\Phi_i(I)=I$.  If  $f:(0,\infty)\to(0,\infty)$ is an  operator log-convex function, then
 $$f\left(\sum_{i=1}^{n}\Phi_i(A_i)\right)\leq\left(\sum_{i=1}^{n}\Phi_i\left(f(A_i)^{-1}\right)\right)^{-1}$$
 for all strictly positive operators $A_1,\cdots,A_n$.
\end{corollary}
\begin{proof}
 Apply  Theorem \ref{tcdjl} to the strictly positive operator $A=A_1\oplus\cdots\oplus A_n$ and the unital positive linear map $\Phi:\mathbb{B}(\mathscr{H}\oplus\cdots \oplus\mathscr{H})\to\mathbb{B}(\mathscr{H})$ defined  by $\Phi(A_1\oplus\cdots\oplus A_n)=\sum_{i=1}^{n}\Phi_i(A_i)$.
\end{proof}

The next result shows that every operator log-convex function is sub-additive.
\begin{proposition}
  If $f:(0,\infty)\to(0,\infty)$ is an operator log-convex function, then $f$ is sub-additive. More precisely
  \begin{align*}
    f(A+B)\leq f(A)\sharp f(B)\leq f(A)+f(B)
  \end{align*}
  for all  strictly positive   operators $A,B$.
\end{proposition}
\begin{proof}
  Assume that $A$ and $B$ are strictly positive operators. Then
  \begin{align*}
    f(A+B)=f((2A)\nabla(2B))&\leq f(2A)\sharp f(2B)\\
    & \leq f(A)\sharp f(B)\ \qquad (\mbox{by (1) of Theorem A}) \\
    &\leq f(A)\nabla f(B)\qquad (\mbox{by the A-G inequality})\\
    &\leq f(A) + f(B).
  \end{align*}
\end{proof}

\section{the non-commutative $f$-divergence functional}
Let $X_1,\cdots,X_n$ and $Y_1,\cdots,Y_n$ be $n$-tuples of positive operators on $\mathscr{H}$. If follows from the jointly operator concavity of the operator geometric mean that
\begin{align}\label{CS}
\sum_{i=1}^{n}X_i\sharp Y_i\leq \left(\sum_{i=1}^{n}X_i\right)\sharp \left(\sum_{i=1}^{n}Y_i\right).
\end{align}
This inequality is known as the   operator version of the Cauchy--Schwarz inequality.

To achieve our result, we need a more general version of \eqref{CS}.  Assume that $(A_t)_{t\in T}$ and $(B_t)_{t\in T}$ are continuous fields of strictly positive operators in $\mathfrak{A}$.  We can generalize \eqref{CS} as
$$ \int_T (A_t\sharp B_t )d\mu(t)\leq \left(\int_T A_t d\mu(t)\right)\sharp \left(\int_T B_td\mu(t)\right).$$

\begin{lemma}\label{CSI}
If  $(A_t)_{t\in T}$ and $(B_t)_{t\in T}$ are  continuous fields of strictly positive operators in $\mathfrak{A}$, then
$$ \int_T (A_t\sharp B_t )d\mu(t)\leq \left(\int_T A_t d\mu(t)\right)\sharp \left(\int_T B_td\mu(t)\right).$$
\end{lemma}
\begin{proof}
  Put $A=\int_TA_td\mu(t)$ and $B=\int_TB_td\mu(t)$. we have
  \begin{align}\label{11}
    \left(B^{-\frac{1}{2}} A B^{-\frac{1}{2}}\right)^{\frac{1}{2}}
    &=\left(\left(\int_TB_sd\mu(s)\right)^{-\frac{1}{2}} \int_TA_td\mu(t) \left(\int_TB_sd\mu(s)\right)^{-\frac{1}{2}}\right)^{\frac{1}{2}}\nonumber\\
    &=\left( \int_T\left(\int_TB_sd\mu(s)\right)^{-\frac{1}{2}}B_t^{\frac{1}{2}}(B_t^{-\frac{1}{2}}A_t B_t^{-\frac{1}{2}})B_t^{\frac{1}{2}} \left(\int_TB_sd\mu(s)\right)^{-\frac{1}{2}}d\mu(t)\right)^{\frac{1}{2}}\nonumber\\
    &= \left( \int_T C_t^*(B_t^{-\frac{1}{2}}A_t B_t^{-\frac{1}{2}})C_td\mu(t)\right)^{\frac{1}{2}},
     \end{align}
      where $C=B_t^{\frac{1}{2}}\left(\int_TB_sd\mu(s)\right)^{-\frac{1}{2}}$  so that $\int_T C_t^*C_t d\mu(t)=I$. It follows from the operator concavity of the function $t^{\frac{1}{2}}$ that
     \begin{align}\label{22}
     &\left( \int_T C_t^*(B_t^{-\frac{1}{2}}A_t B_t^{-\frac{1}{2}})C_td\mu(t)\right)^{\frac{1}{2}}\nonumber\\
         &\geq \int_T C_t^*\left(B_t^{-\frac{1}{2}}A_t B_t^{-\frac{1}{2}}\right)^{\frac{1}{2}}C_td\mu(t) \qquad (\mbox{by the operator Jensen inequality })\nonumber\\
    &=  \left(\int_TB_sd\mu(s)\right)^{-\frac{1}{2}}\int_T B_t^{\frac{1}{2}}\left(B_t^{-\frac{1}{2}}A_t B_t^{-\frac{1}{2}}\right)^{\frac{1}{2}}B_t^{\frac{1}{2}}d\mu(t) \left(\int_TB_sd\mu(s)\right)^{-\frac{1}{2}}.
  \end{align}
It follows that from \eqref{11} and \eqref{22} that
\begin{align*}
  \left(B^{-\frac{1}{2}} A B^{-\frac{1}{2}}\right)^{\frac{1}{2}}\geq
  B^{-\frac{1}{2}}\left(\int_T B_t^{\frac{1}{2}}\left(B_t^{-\frac{1}{2}}A_t B_t^{-\frac{1}{2}}\right)^{\frac{1}{2}}B_t^{\frac{1}{2}}d\mu(t)\right) B^{-\frac{1}{2}}
\end{align*}
from which we get the desired result.
\end{proof}
Now we present a property of the non-commutative $f$-divergence functional of an operator log-convex function.
\begin{theorem}\label{main}
   A continuous  function $f:(0,\infty)\to(0,\infty)$ is operator log-convex function if and only if
   \begin{align}\label{00}
  \Theta\left( \widetilde{A}\nabla\widetilde{C},\widetilde{B}\nabla\widetilde{D}\right)\leq \left(\Theta\left( \widetilde{A},\widetilde{B}\right)\nabla \Theta\left( \widetilde{A},\widetilde{D}\right)\right)\sharp \left(\Theta\left( \widetilde{C},\widetilde{B}\right)\nabla \Theta\left( \widetilde{C},\widetilde{D}\right)\right)
\end{align}
for all continuous fields $\widetilde{A}=(A_t)_{t\in T},\widetilde{B}=(B_t)_{t\in T}, \widetilde{C}=(C_t)_{t\in T}$ and $\widetilde{D}=(D_t)_{t\in T}$ of strictly positive operators in $\mathfrak{A}$.
\end{theorem}
\begin{proof}
First we show that if $f$ is  operator log-convex  on $(0,\infty)$, then $\Theta$ is  operator log-convex in its first variable. Assume that $\widetilde{A}=(A_t)_{t\in T},\widetilde{B}=(B_t)_{t\in T}, \widetilde{C}=(C_t)_{t\in T}$ and $\widetilde{D}=(D_t)_{t\in T}$ are continuous fields of strictly positive  operators in $\mathfrak{A}$. For every $t\in T$
\begin{align}\label{1}
  f\left(B_{t}^{-\frac{1}{2}} (A_{t}\nabla C_{t})B_{t}^{-\frac{1}{2}}\right)&=f\left(\left(B_{t}^{-\frac{1}{2}}A_{t}B_{t}^{-\frac{1}{2}} \right)\nabla \left(B_{t}^{-\frac{1}{2}} C_{t}B_{t}^{-\frac{1}{2}}\right)\right)\nonumber\\
  &\leq f\left(B_{t}^{-\frac{1}{2}}A_{t}B_{t}^{-\frac{1}{2}} \right)\sharp  f\left(B_{t}^{-\frac{1}{2}}C_{t}B_{t}^{-\frac{1}{2}} \right),
\end{align}
where we use the operator log-convexity of $f$. Multiplying both sides of \eqref{1} by $B_{t}^{\frac{1}{2}}$ we get
\begin{align}\label{2}
  B_{t}^{\frac{1}{2}}f\left(B_{t}^{-\frac{1}{2}} (A_{t}\nabla C_{t})B_{t}^{-\frac{1}{2}}\right)B_{t}^{\frac{1}{2}}
  &\leq B_{t}^{\frac{1}{2}}\left(f\left(B_{t}^{-\frac{1}{2}}A_{t}B_{t}^{-\frac{1}{2}} \right)\sharp  f\left(B_{t}^{-\frac{1}{2}}C_{t}B_{t}^{-\frac{1}{2}} \right)\right)B_{t}^{\frac{1}{2}}\\
  &= B_{t}^{\frac{1}{2}}f\left(B_{t}^{-\frac{1}{2}}A_{t}B_{t}^{-\frac{1}{2}} \right) B_{t}^{\frac{1}{2}}\sharp  B_{t}^{\frac{1}{2}}f\left(B_{t}^{-\frac{1}{2}}C_{t}B_{t}^{-\frac{1}{2}} \right)B_{t}^{\frac{1}{2}}.\nonumber
\end{align}
The last equality follows from the property of (geometric)  means.
Integrating \eqref{2} over $T$ and using Lemma \ref{CSI} we obtain
\begin{align*}
  \int_T &B_{t}^{\frac{1}{2}}f\left(B_{t}^{-\frac{1}{2}} (A_{t}\nabla C_{t})B_{t}^{-\frac{1}{2}}\right)B_{t}^{\frac{1}{2}}d\mu(t)\nonumber\\
    &\leq  \int_T \left(B_{t}^{\frac{1}{2}}f\left(B_{t}^{-\frac{1}{2}}A_{t}B_{t}^{-\frac{1}{2}} \right) B_{t}^{\frac{1}{2}}\sharp  B_{t}^{\frac{1}{2}}f\left(B_{t}^{-\frac{1}{2}}C_{t}B_{t}^{-\frac{1}{2}} \right)B_{t}^{\frac{1}{2}}\right)d\mu(t)\quad\qquad(\mbox{by \eqref{2}})\nonumber\\
    &\leq \left( \int_T B_{t}^{\frac{1}{2}}f\left(B_{t}^{-\frac{1}{2}}A_{t}B_{t}^{-\frac{1}{2}} \right) B_{t}^{\frac{1}{2}}d\mu(t) \right)\sharp  \left( \int_T B_{t}^{\frac{1}{2}}f\left(B_{t}^{-\frac{1}{2}}C_{t}B_{t}^{-\frac{1}{2}} \right)B_{t}^{\frac{1}{2}}d\mu(t)\right),
    \end{align*}
    i.e.,
\begin{align}\label{44}
\Theta\left( \widetilde{A}\nabla\widetilde{C},\widetilde{B}\right)\leq\Theta\left( \widetilde{A},\widetilde{B}\right)\sharp\Theta\left(\widetilde{C},\widetilde{B}\right).
\end{align}
Therefore
\begin{align*}
  \Theta\left( \widetilde{A}\nabla\widetilde{C},\widetilde{B}\nabla\widetilde{D}\right)&\leq \Theta\left( \widetilde{A},\widetilde{B}\nabla\widetilde{D}\right)\sharp\Theta\left(\widetilde{C},\widetilde{B}\nabla\widetilde{D}
  \right)\qquad(\mbox{by \eqref{44}})\\
  &\leq \left(\Theta\left( \widetilde{A},\widetilde{B}\right)\nabla \Theta\left( \widetilde{A},\widetilde{D}\right)\right)\sharp \left(\Theta\left( \widetilde{C},\widetilde{B}\right)\nabla \Theta\left( \widetilde{C},\widetilde{D}\right)\right).
\end{align*}
The last inequality follows from the joint operator convexity of $\Theta$ \cite{MK} and  monotonicity   of operator means.

Assume for the converse that $\Theta$  satisfies \eqref{00}.  Let $T=\{1\}$ and
$\mu$ be the counting measure on $T$. Let $A$ and $C$ be strictly positive operators in $\mathfrak{A}$. Then
$$f(A\nabla C)=\Theta(A\nabla C, I)\leq \Theta(A,I)\sharp \Theta(C,I)=f(A)\sharp f(C),$$
which means that $f$ is operator log-convex.
\end{proof}

\begin{remark}
  Infact  Theorem \ref{main} assert that $f$ is operator log-convex if and only if the non-commutative $f$-divergence functional $\Theta$ is operator log-convex in its first variable and operator convex in its second variable.
\end{remark}
\begin{corollary}\label{co1}
A continuous non-negative function $f$ is operator log-convex function if and only if the associated perspective function $g$ is operator log-convex function in its first variable and operator convex in its second variable.
\end{corollary}
The next theorem provides a Choi--Davis--Jensen type inequality for perspectives of operator log-convex functions.
\begin{theorem}\label{th3}
 Let  $f:(0,\infty)\to(0,\infty)$  be a continuous  function and $g$ be the associated perspective function. Then $f$ is operator log-convex  if and only if
{\small \begin{align}\label{33}
g\left(\int_T\Phi_t(A_t\nabla C_t)d\mu (t),\int_T\Phi_t(B_t)d\mu(t)\right)\leq
 \left(\int_T\Phi_t(g(A_t,B_t))d\mu(t)\right)\sharp\left(\int_T\Phi_t(g(C_t,B_t))d\mu(t)\right)
 \end{align}}
 for all unital fields $\widetilde{\Phi}=(\Phi_t)_{t\in T}:\mathfrak{A}\to\mathfrak{B}$  of  positive linear maps and all  continuous fields $\widetilde{A}=(A_t)_{t\in T}, \widetilde{B}=(B_t)_{t\in T}$  and $\widetilde{C}=(C_t)_{t\in T}$ of strictly positive operators in $\mathfrak{A}$.
\end{theorem}
\begin{proof}
 Assume that $f$ is operator log convex. Then $f$ is operator convex and so $g$ is jointly operator convex
 \cite{E,G}. Put $X=\int_T\Phi_s(B_s)d\mu(s)$ and let the continuous filed of positive linear mappings  $(\Psi_t)_{t\in T}:\mathfrak{A}\to\mathfrak{B}$ be defined by
 $$\Psi_t(Y)=X^{\frac{-1}{2}}
 \Phi_t\left(B_{t}^{\frac{1}{2}}YB_{t}^{\frac{1}{2}}\right)X^{\frac{-1}{2}}$$
 so that $\int_T\Psi_t(I)d\mu(t)=I$. Therefore
 \begin{align*}
  & g\left(\int_T\Phi_t(A_t\nabla C_t)d\mu (t),\int_T\Phi_t(B_t)d\mu(t)\right)=  X^{\frac{1}{2}}f\left( X^{\frac{-1}{2}}\int_T \Phi_t(A_t\nabla C_t)d\mu (t)X^{\frac{-1}{2}}\right) X^{\frac{1}{2}}\\
&= X^{\frac{1}{2}}f\left(\int_T \Psi_t\left( B_t^{\frac{-1}{2}}(A_t\nabla C_t) B_t^{\frac{-1}{2}}\right)
d\mu (t)\right) X^{\frac{1}{2}}\\
&\leq X^{\frac{1}{2}}\left(\int_T \Psi_t\left(f\left( B_t^{\frac{-1}{2}}(A_t\nabla C_t) B_t^{\frac{-1}{2}}\right)\right)d\mu (t)\right) X^{\frac{1}{2}}\quad(\mbox{ by the Jensen operator inequality})\\
&=\int_T\Phi_t\left(B_t^{\frac{1}{2}}f\left(B_t^{\frac{-1}{2}}(A_t\nabla C_t)B_t^{\frac{-1}{2}}\right)B_t^{\frac{1}{2}}\right)d\mu(t)\\
&= \int_T\Phi_t(g(A_t\nabla C_t,B_t))d\mu(t)\\
&\leq \int_T\Phi_t (g(A_t,B_t)\sharp g(C_t,B_t))d\mu(t) \qquad(\mbox{by Corollary \ref{co1}})\\
&\leq \int_T \Phi_t (g(A_t,B_t))\sharp \Phi_t(g(C_t,B_t))d\mu(t)\qquad(\mbox{by operator concavity of $\sharp$})\\
&\leq \left(\int_T \Phi_t (g(A_t,B_t))d\mu(t)\right)\sharp \left(\int_T \Phi_t (g(C_t,B_t))d\mu(t)\right)\quad(\mbox{by Lemma \ref{CSI}}).
 \end{align*}
 For the converse, put  $T=\{1\}$ and let $\mu$ be the counting measure on $T$. If $A$ and $C$ are strictly positive, then with $\Phi(A)=A$ and $B=I$, inequality \eqref{33} implies  the operator log-convexity of $f$.
\end{proof}
\begin{example}
  Let  the operator log-convex function $f:(0,\infty)\to(0,\infty)$ be defined by $f(t)=t^{-1}$. It follows from Theorem \ref{th3} that
  \begin{align*}
    g(\Phi(A\nabla C),\Phi(B))\leq\Phi(g(A,B))\sharp\Phi(g(C,B))
  \end{align*}
  or equivalently
   \begin{align*}
    \Phi(B)\Phi(A\nabla C)^{-1}\Phi(B)\leq\Phi(BA^{-1}B)\sharp\Phi(BC^{-1}B).
  \end{align*}
  Therefore
\begin{align}\label{exam}
\Phi(A\nabla C)^{-1}&\leq  \Phi(B)^{-1}\left(\Phi(BA^{-1}B)\sharp\Phi(BC^{-1}B)\right)\Phi(B)^{-1}\\
&=\Phi(B)^{-1}\Phi(BA^{-1}B)\Phi(B)^{-1}\sharp\Phi(B)^{-1}\Phi(BC^{-1}B)\Phi(B)^{-1}.\nonumber
\end{align}
  Note that it  follows from the operator convexity of $f$ that
 \begin{align*}
\Phi(A\nabla C)^{-1}&=\left(\frac{\Phi(A)+\Phi(C)}{2}\right)^{-1}\\
&\leq \frac{\Phi(A)^{-1}+\Phi(C)^{-1}}{2}\qquad(\mbox{by  operator convexity of $f(t)=t^{-1}$})\\
&\leq \frac{\Phi\left(A^{-1}\right)+\Phi\left(C^{-1}\right)}{2}\qquad(\mbox{by  operator convexity of $f(t)=t^{-1}$}),
\end{align*}
while  with $B=I$, inequality \eqref{exam} provides a sharper inequality:
$$\Phi(A\nabla C)^{-1}\leq \Phi\left(A^{-1}\right)\sharp\Phi\left(C^{-1}\right)\leq\frac{\Phi\left(A^{-1}\right)+\Phi\left(C^{-1}\right)}{2}.$$

\end{example}
\begin{corollary}\label{co2}
 Let $A,B,C$ be strictly positive operators. If $f$ is an operator log-convex function and $g$ is its perspective function, then
 $$g\left(\left\langle\frac{A+C}{2}x,x\right\rangle,\langle Bx,x\rangle\right)\leq \sqrt{\langle g(A,B)x,x\rangle\langle g(C,B)x,x\rangle}$$
\end{corollary}
for every unit vector $x$.
\begin{example}
  Applying Corollary \ref{co2} to the operator log-convex function $f(t)=t^{-1}$ defined on $(0,\infty)$ we get
 \begin{align*}
 \langle  B x,x\rangle\left\langle \frac{A+C}{2}x,x\right\rangle^{-1}\langle Bx,x\rangle &\leq \sqrt{\langle  BA^{-1}Bx,x\rangle\langle  BC^{-1}Bx,x\rangle}\\
 &\leq \left\langle B\left(\frac{A^{-1}+C^{-1}}{2}\right)Bx,x\right\rangle
 \end{align*}
\end{example}
for every unit vector $x$.

\bibliographystyle{amsplain}

\end{document}